\documentclass[final,leqno,onefignum,onetabnum]{siamltex1213}

\usepackage[latin1]{inputenc}
\usepackage{amssymb}
\usepackage{amsmath}
\usepackage{esint}
\usepackage{hyperref}
\usepackage{multirow}
\usepackage{multicol}
\usepackage{graphics}
\usepackage{epsfig}
\usepackage{color}

\newcommand{\rn}{{\mathbb{R}^n}}
\newcommand{\phii}{\varphi}
\newcommand{\pp}{\partial}

\newcommand{\W}{\Omega}
\newcommand{\eps}{\varepsilon}

\def\R{{\mathbb {R}}}

\def\T{{\mathcal {T}}}
\def\V{{\mathbb {V}}}

\newtheorem{remark}[theorem]{Remark}

\title{A fractional Laplace equation:  regularity of solutions and Finite Element approximations
\thanks{This work has been partially supported by CONICET under grant PIP 2014-2016 11220130100184CO}}

\author{Gabriel Acosta \and Juan Pablo Borthagaray  \thanks{ IMAS - CONICET and Departamento de Matem\'a\-tica, FCEyN - Universidad de Buenos Aires, Ciudad Universitaria, Pabell\'on I  (1428) Buenos Aires, Argentina. 
(\email{gacosta@dm.uba.ar}, \email{jpbortha@dm.uba.ar}).}
}

\begin{document} 
\maketitle
\slugger{mms}{xxxx}{xx}{x}{x--x}

\begin{abstract}
This paper deals with the \emph{integral}   version of the Dirichlet homogeneous fractional Laplace equation. For this problem  weighted and fractional Sobolev a priori estimates are provided in terms of the H\"older regularity of the data. By relying on these results, optimal order of convergence for the standard linear finite element method is proved for quasi-uniform as well as graded meshes. Some numerical examples
are given showing results in agreement with the theoretical predictions.
\end{abstract}



\pagestyle{myheadings}
\thispagestyle{plain}
\markboth{GABRIEL ACOSTA AND JUAN PABLO BORTHAGARAY}{ESTIMATES FOR THE FE SOLUTION OF A FRACTIONAL LAPLACIAN}

%\maketitle
\section{Introduction} \label{sec:intro}
In the last years, the study of nonlocal operators has been an active area of research in different branches of mathematics, and these operators have been employed to model problems in which different length scales are involved. Anomalous diffusion phenomena are ubiquitous in nature \cite{Klafter, MetzlerKlafter}; among the several applications of these nonlocal models let us mention image processing \cite{BuadesColl, GattoHesthaven, GilboaOsher, YifeiZhang}, finance \cite{CarrHelyette, RamaTankov}, electromagnetic fluids \cite{McCayNarasimhan}, peridynamics \cite{Silling} and porous media flow \cite{BensonWheatcraft, CushmanGinn}.

In this work we will be interested in the fractional Laplace operator of order $s$, which we will denote by $(-\Delta)^s$ and simply call the fractional Laplacian. In the theory of stochastic processes, this operator appears as the infinitesimal generator of a stable Lévy process \cite{Bertoin, Valdinoci}.

If the domain under consideration is the whole space $\rn$, then there is a natural way to define it as a pseudodifferential operator of symbol $|\xi|^{2s}$. Indeed, for a function $u$ in the Schwartz class $\mathcal{S}$, let
\begin{equation}
(-\Delta)^s u = \mathcal{F}^{-1} \left( |\xi|^{2s} \mathcal{F} u \right) ,
\label{eq:fourier}
\end{equation}
where $\mathcal{F}$ denotes the Fourier transform. It is expectable for this operator to approximate the usual Laplacian for $s\to1$ and the identity as $s\to0$.

The fractional Laplacian can equivalently be defined by means of the identity \cite{Hitchhikers}
\begin{equation}
(-\Delta)^s u (x) = C(n,s) \mbox{ P.V.} \int_\rn \frac{u(x)-u(y)}{|x-y|^{n+2s}},
\label{eq:fraccionarioyo}
\end{equation}
where the normalization constant 
$$C(n,s) = \frac{2^{2s} s \Gamma(s+\frac{n}{2})}{\pi^{n/2} \Gamma(1-s)} $$
is taken in order to be consistent with definition \eqref{eq:fourier}.

One of the main difficulties arising in the study of this operator is its nonlocality; in order to localize it, Caffarelli and Silvestre \cite{CaffarelliSilvestre} showed that it can be realized as a Dirichlet-to-Neumann operator by means of an extension problem in the half-space $\mathbb{R}^{n+1}_+$.

However, there is more than one way to define the fractional Laplacian on an open bounded set $\W$: 
\begin{enumerate}
	\item One first possibility is to consider fractional powers of the Dirichlet Laplace operator in the sense of spectral theory. Indeed, let $\{\psi_k, \lambda_k \}_{k \in \mathbb{N}} \subset H^1_0 (\W) \times \mathbb{R}_+ $ be the set of normalized eigenfunctions and eigenvalues for the Laplace operator in $\W$ with homogeneous Dirichlet boundary conditions, so that $\{\psi_k\}_{k \in \mathbb{N}}$ is an orthonormal basis of $L^2(\W)$ and 
$$ \left\lbrace
  \begin{array}{c l}
      -\Delta \psi_k = \lambda_k \psi_k  & \mbox{ in }\W, \\
      \psi_k = 0 & \mbox{ on }\partial\W .\\
      \end{array}
    \right.
$$
Then, the \emph{spectral} fractional Laplacian $(-\Delta)^s_S$ is defined for $u \in C^\infty_0(\W)$ by
$$ (-\Delta)^s_S u = \sum_{k=1}^\infty \langle u , \psi_k \rangle \lambda_k^s u_k ,$$
and can be subsequently extended by density to the Hilbert space $H^s(\W)$. In \cite{StingaTorrea}, the Caffarelli-Silvestre result was proved for this operator, thus achieving a local problem posed on a semi-infinite cylinder $\W\times(0,\infty)$. This localization was exploited by Nochetto, Otárola and Salgado in \cite{NochettoOtarola}, 
where the authors study the numerical approximation of the spectral fractional Laplacian by considering graded meshes in the extended variable.

	\item A second feasible definition is attained by considering the integral formulation \eqref{eq:fraccionarioyo}, and restricting it to functions supported in $\W$. This gives rise to the \emph{integral} fractional Laplacian  $(-\Delta)^s_I u$. This operator is different to the spectral fractional Laplacian; for example, their difference is positive definite and positivity preserving \cite{MusinaNazarov}. See also	\cite{ServadeiValdinoci}, where the spectra of these operators are compared. The main difficulties to overcome when dealing with numerical analysis of this integral fractional Laplacian are associated to its nonlocality and to the singularity at $x=y$ of the kernel it involves.
	
	\item Finally, it is also possible to consider a \emph{regional} fractional Laplacian $(-\Delta)^s_R$, in which the integration in \eqref{eq:fraccionarioyo} is restricted to $\W$. This operator is known to be the infinitesimal generator of the so-called censored stable Lévy processes \cite{Bogdan}. 
\end{enumerate}

Throughout this work, we will restrict to the second definition, and for sake of simplicity we will skip the $I$ subindex when denoting it. The analysis we will perform is valid for the regional fractional Laplacian as well. Moreover, it can be straightforwardly extended to operators of the form 
$$ \mathcal{L_K} u (x) = \mbox{P.V.} \int_\rn (u(x)-u(y)) \, K(x-y) \, dy , $$
where the kernel $K:\rn \setminus \{ 0\} \to (0, \infty)$ satisfies
\begin{align*}
& \gamma K \in L^1(\rn) , \mbox{ with } \gamma(x) = \min \{ 1, |x|^2 \}, \\
& \exists \theta > 0, \, s\in(0,1) \mbox{ such that } K(x) \geq \theta |x|^{-(n+2s)}, \, x \in \rn \setminus \{ 0 \}, \\
& K(x) = K(-x), \, x \in \rn \setminus \{ 0 \}. 
\end{align*}

Numerical approximation for the fractional Laplacian on bounded domains has been addressed in the last years. D'Elia and Gunzburger \cite{DEliaGunzburger} exploited the nonlocal vector calculus introduced in \cite{DuGunzburger} in order to perform a study of convergence of certain approximations of the fractional Laplacian as the nonlocal interactions become infinite. Huang and Oberman \cite{HuangOberman} proposed a method which combines finite differences with numerical quadrature, obtained a discrete convolution operator and studied numerically the convergence and order of their method in the $L^\infty(\W)$ norm. However, these methods were implemented only in 1 dimension, and regularity of solutions in those works is assumed as part of the hypotheses.

This work is an attempt to deal with basic 
analytical aspects required to convey a 
complete finite element analysis of the 
following fractional Laplace problem  
\begin{equation}
\left\lbrace
  \begin{array}{l}
      (-\Delta)^s u = f  \mbox{ in }\W, \\
      u = 0 \mbox{ in }\W^c , \\
      \end{array}
    \right.
\label{eq:fraccionario}
\end{equation}
where $0<s<1$ and $(-\Delta)^s u$ denotes the operator defined  in \eqref{eq:fraccionarioyo}.
Our aim is to provide regularity of solutions as well as a priori error estimates for the discrete approximations together with a feasible
finite element implementation. 

From now on  we assume that  $\W \subset \rn$  is a bounded  Lipschitz domain and $f$ a bounded function defined in $\W$.

The fractional Laplace equation in the form \eqref{eq:fraccionario} shares some key analytical
features of the classical Laplacian  making it  amenable - in principle -  to a direct finite
element treatment. Nevertheless,  Sobolev  regularity results for this problem are recent and 
expressed in terms of H\"ormander spaces \cite{Grubb}. Moreover, in that paper the domain
is required to have a $C^\infty$ boundary, which makes its results not entirely
satisfactory for a  FE analysis. Also, some numerical difficulties -such as  
handling the singularity of the kernel- seem to be the main disadvantages that 
have discouraged a direct FE approach. Concerning the latter, we found   that applying rather standard techniques (actually borrowed from the theory of the Boundary Element Method  \cite{SauterSchwab}) together with  an appropriate treatment of the integrals 
involving the unbounded domain $\Omega^c$, accurate FEM solutions 
(at least in 2D) could  be delivered.

%up to the boundary of the domain (for interior regularity, see the recent 
%preprint \cite{Cozzi})

In this paper we address Sobolev 
regularity results for \eqref{eq:fraccionario} in less regular domains. 
We provide  weighted fractional
Sobolev a priori estimates in terms of the H\"older regularity of the data
for Lipschitz domains satisfying the exterior ball condition.  
As it is shown in Remark \ref{rem:sharp}, our predicted regularity is sharp. 
The proof relies on recent H\"older regularity results of Ros-Oton and Serra \cite{RosOtonSerra} 
(see Section \ref{sec:regularity}). 
Even though Sobolev-Sobolev (instead Sobolev-H\"older)  results are preferable, this theorem is a key tool within the  FE analysis developed along this work.

The nonlocal nature of problem \eqref{eq:fraccionario} is reflected in the fact that fractional Sobolev norms are not additive respect to the domains. Therefore,
after finding a suitable interpolation operator and getting adequate local
interpolation estimates, some cautions must be taken in order to bound the global interpolation error \cite{Faermann2, Faermann}. In order to deal with graded meshes we extend well known error estimates for the Scott-Zhang interpolation operator to our weighted fractional Sobolev spaces. These estimates are derived by introducing \emph{improved} Poincar\'e inequalities in the fractional setting in a form that we were unable to find in the literature.

The organization of this paper is as follows. Section \ref{sec:setting} is 
devoted to basic definitions and properties of the operator and spaces involved.
In Section \ref{sec:regularity}, the matter of regularity of solutions is 
addressed; it is proved that under some Hölder regularity hypotheses on the 
source $f$, the solutions of the problem under consideration gain half a 
derivative in the Sobolev sense. Moreover, if the source is more regular, sharper
estimates in weighted fractional spaces are shown to hold. Then, in 
Section \ref{ss:fem} interpolation estimates are developed both in standard and
weighted spaces, thus obtaining a priori error estimates for Finite Element (FE) 
solutions of equation \eqref{eq:fraccionario}. Finally, in 
Section \ref{sec:implementation}, some remarks on the implementation are 
discussed and several numerical tests are presented showing results in complete
agreement with our theoretical predictions. More details about the 
implementation in dimension two can be found in the forthcoming paper \cite{ABB}, 
and related results for the fractional eigenvalue problem in \cite{BdPM}.

\subsection*{Acknowledgments} The authors would like to thank L. Del Pezzo and S. Mart\'inez for stimulating discussions on the topic and valuable help by thoroughly reading draft versions of this paper.

%%%%%%%%%%%%%%%%%%%%%%%%%%%%%%%%%%%%%%%%%%%%%%%%%%%%%%%%%%%%%%%%%%%%%%%%%%%%%%%%%%%%%%%%%%%%%%%%%%%%%%%%%%%%%%%%%%%%%%%%%%%%%
\section{Analytic setting}  \label{sec:setting}
In this section we set the notation and review some properties of the spaces involved in the rest of the paper. Let us begin by recalling some function spaces. Throughout this work, $s$ is a parameter such that $0<s<1$ and $\W \subset \rn$ ($n\geq 1$) is a bounded Lipschitz domain.  

\subsection{Function spaces}
The Sobolev space $H^s(\W)$ is defined by
$$ H^s(\W) = \left\{ v \in L^2 (\W) \colon |v|_{H^s(\W)} < \infty \right\} ,$$
where 
$$|v|_{H^s(\W)} = \left(\iint_{\W\times\W} \frac{(v(x)-v(y))^2}{|x-y|^{n+2s}} dx \, dy \right)^{\frac{1}{2}}$$
denotes the Aronszajn-Slobodeckij seminorm. The space $H^s(\W)$ is a Hilbert space, equipped with the norm
$$\| v \|_{H^s(\W)} = \|v \|_{L^2(\W)} + |v|_{H^s(\W)} .$$
Equivalently, the space $H^s(\W)$ may be regarded as the restriction to $\W$ of functions in $H^s(\rn)$. Zero trace spaces ${H^s_0(\W)}$ can be defined as the closure of $C^\infty_c(\W)$ w.r.t. the ${H^s}$ norms. Equivalently, if the boundary of the domain is smooth, they can be defined through real interpolation of spaces by the $K$-method (for example, \cite[Chapter 1]{LionsMagenes}). It is well-known that for
$0< s \le 1/2$, the identity $H_0^s(\W)=H^s(\W)$ holds.

Sobolev spaces of order grater than 1 are defined in the following way: given $k \in \mathbb{N}$, then
$$H^{k+s} (\W) = \left\{ v \in H^k(\W) \colon |D^\alpha v | \in H^s(\W) \, \forall \alpha \mbox{ s.t. } |\alpha| = k \right\} ,$$
furnished with the norm
$$ \| v \|_{H^{k+s} (\W)} = \| v \|_{H^k(\W)} + \sum_{|\alpha| = k } | D^\alpha v |_{H^s(\W)}.$$
In the sequel we assume that $k=0$ or $k=1$, which are the cases of interest along our presentation. 

Let us recall that weighted
Sobolev spaces are  a customary tool for dealing with singular solutions. In the present context we find useful  some fractional and weighted spaces. Our
weights are powers of the distance to the boundary of $\Omega$. Accordingly, we introduce the notation  $\delta : \R^n \to \R_{\ge 0}$, $\delta(x)=d(x, \pp \W)$ and consider the norm    
$$\| v \|_{H^{1+s}_\alpha (\W)}^2 = \| v \|_{H^1 (\W)}^2 + \sum_{|\beta| = 1 } 
\iint_{\W\times\W} \frac{|D^\beta v(x)-D^\beta v(y)|^2}{|x-y|^{n+2s}} \, \delta(x,y)^{2\alpha} dx \, dy $$
where $\alpha \ge 0$, $s \in (0,1)$
and   
$$\delta(x,y) = \min \{ \delta(x), \delta(y) \}.$$ 
In this way we write
$$H^\ell_\alpha (\W) = \left\{ v \in H^1(\W) \colon \| v \|_{H^\ell_\alpha (\W)} < \infty \right\} .$$
Global versions $H^\ell_{\alpha,\W} (\rn)$ are easily obtained  integrating in the whole space $\rn$ and taking $\delta$ as before.  In the sequel we drop the reference to $\W$ in the global case and write $H^\ell_{\alpha} (\rn)$.

\begin{remark}
 \label{rem:pesosA2}
Although  we are interested in the case $ \alpha\ge 0$,  let us recall that in the definition of standard weighted Sobolev spaces $H^k_\alpha(\W)$, with $k$ being a nonnegative integer, arbitary powers of $\delta(x)$    can be considered \cite{Kufner}. On the other hand, for general  weights some restrictions must be taken into account in order to get a right definition of the spaces. A classical family of weights is that of the Muckenhoupt $A_2$  class \cite{Stein}. In the global version
$H^\ell_{\alpha} (\rn)$ we restrict the range of $\alpha$ to $0\le \alpha < 1/2$ in order to have $\delta^\alpha\in A_2$. Moreover, this restriction arises naturally in our results involving graded meshes (see Remark \ref{rem:mallas}).
\end{remark}

Weak solutions of equation \eqref{eq:fraccionario} may be defined by multiplying by a test function and integrating by parts the Laplacian term. Namely, consider the space
$$\V = \{ u \in H^s(\rn) \, \colon \, u = 0 \mbox{ in } \W^c \},$$
equipped with the $H^s(\rn)$ norm. For $\V$
it is known that $C_0^\infty(\Omega)$ is a dense set \cite{FiscellaServadeiValdinoci}. Then, the weak formulation of \eqref{eq:fraccionario} is: find $u \in \V$ such that
\begin{equation}
\frac{C(n,s)}{2} \iint_Q \frac{(u(x)-u(y))(v(x)-v(y))}{|x-y|^{n+2s}} dx \, dy = \int_\W f(x) \, v(x) \, dx
\label{eq:cont}
\end{equation}
for all $v \in \V$, where $Q= (\W \times \rn) \cup (\rn \times \W)$. The integral in the right hand side of the previous equation makes sense if $f$ belongs to the dual space of $\V$, which we denote by $\V^*$. Observe that $(H^s(\rn))^* \subset \V^*$.

At this point we recall for  further reference the following result  \cite{BrennerScott}.  
\begin{proposition}[Poincaré inequality I] \label{prop:poincareII} Let $S$ be an star-shaped domain w.r.t. a ball $B$. For any $u\in H^s(S)$
with $0<s<1$ we call $\bar v=\frac{1}{|S|}\int_Sv$. Then we have 
$$ \| v - \bar v\|_{L^2(S)} \leq c d_S^s |v|_{H^s(S)},$$
with $c$ bounded in terms of $\frac{d_S}{d_B}$,   where $d_S=diam(S)$  and $d_B=diam(B)$.
\end{proposition}
\begin{proof}
We write
$$
\int_S (v-\bar v)^2 dx=\frac{1}{|S|^2}\int_S\left(
\int_S (v(x)-v(y))dy\right)^2dx \le
\frac{1}{|S|}\int_S\int_S\left(
 v(x)-v(y)\right)^2dydx, 
$$
therefore
$$
\int_S (v-\bar v)^2 dx\le \frac{d_S^{n+2s}}{|S|}
\int_S\int_S\frac{\left(
 v(x)-v(y)\right)^2}{|x-y|^{n+2s}}dydx.
$$
Taking into account that $d_B^n\sim |B|\le |S|$,
the Proposition follows. \qquad
\end{proof}

\begin{remark}
Following \cite{BrennerScott} we call $\frac{d_S}{d_B}$ the chunkiness parameter of $S$. 
\end{remark}
Another  well-known result is the following. We include its proof here for completeness. 
\begin{proposition}[Poincaré inequality II] \label{prop:poincare} Given $\W$ as above, there exists a constant $c=c(\W,n,s)$ such that 
$$ \| v \|_{L^2(\W)} \leq c |v|_{H^s(\rn)}  \quad \forall v \in \V $$
\end{proposition}
\begin{proof}
By Lemma 6.1 of \cite{Hitchhikers}, there exists some constant $c(n,s)>0$ such that for all $x \in \W$,
$$c(n,s) |\W|^{-\frac{2s}{n}} \leq \int_{\W^c} \frac{1}{|x-y|^{n+2s}} dy  .$$
On the other hand, since $v \equiv 0$ on $\W^c$ we know that $v(x)^2 = (v(x)-v(y))^2$ for all $x \in \W, \, y \in \W^c$. So, we obtain
$$c(n,s) |\W|^{-\frac{2s}{n}} \int_\W v(x)^2 dx \leq \iint_{\W\times\W^c} \frac{(v(x)-v(y))^2}{|x-y|^{n+2s}} dx \, dy ,$$
and the Poincaré inequality follows straightforwardly. \qquad
\end{proof}

An immediate consequence of Proposition \ref{prop:poincare} is that the bilinear form $a: \V \times \V \to \mathbb{R}$, 
\begin{equation}
a(u,v) = \frac{C(n,s)}{2} \iint_Q \frac{(u(x)-u(y))(v(x)-v(y))}{|x-y|^{n+2s}} \, dx \, dy
\label{eq:bilineal}
\end{equation}
is coercive. Its continuity is an obvious consequence of the Cauchy-Schwarz inequality. Therefore, it can be easily seen, by applying the Lax-Milgram  theorem, that if $f \in \V^*$ then there exists a unique $u \in \V$ solution of problem \eqref{eq:cont}.

As the $H^s(\rn)$ seminorm is equivalent to the $H^s(\rn)$ norm on $\V$, let us define
$$\| v \|_\V := a(v,v)^\frac12 = \sqrt{\frac{C(n,s)}{2}} \ |v|_{H^s(\rn)} .$$

The approach we shall follow to obtain error estimates is simply to consider an adequate interpolator in a FE space $\V_h$, and estimate the interpolation error. In order to achieve this, it is convenient to understand the relation between the norm in $\V$, and the (semi)norm in $H^s(\W)$.

\begin{proposition}[Hardy-type inequalities, see \cite{Grisvard, Dyda}] Let $\W$ be as above. If $0<s<\frac{1}{2}$, then there exists $c=c(\W,n,s)>0$ such that 
\begin{equation} \int_\W \frac{|u(x)|^2}{\delta(x)^{2s}} \, dx \, \leq c \| u\|_{H^s(\W)}^2 \ \forall \, u \in H^s(\W).
\label{eq:hardychico} \end{equation}
If $\frac{1}{2}<s<1$, then there exists $c=c(\W,n,s)>0$ such that 
\begin{equation*}
\int_\W \frac{|u(x)|^2}{\delta(x)^{2s}} \, dx \, \leq c |u|_{H^s(\W)}^2 \ \forall \, u \in C_0(\W).
\end{equation*}
\end{proposition}

As a consequence of the previous Proposition we get the following
\begin{corollary} \label{cor:norma} If $0<s<\frac{1}{2}$, then there exists a constant $c=c(\W,n,s)>0$ such that 
\begin{equation*}
\| v \|_\V \leq {c} \| v \|_{H^s(\W)}  \quad \forall v \in \V.
 \end{equation*}
On the other hand, if $\frac{1}{2}<s<1$ there exists a constant ${c}={c}(\W,n,s)>0$ such that  
\begin{equation*}
\| v \|_\V \leq {c} | v |_{H^s({\W})}  \quad \forall v \in \V 
\end{equation*}
\end{corollary}
\begin{proof}
Let us prove the first inequality; the second one follows in the same fashion  recalling that
we can assume  $v\in C_0^\infty(\W)$ \cite{FiscellaServadeiValdinoci} and concluding
by density arguments. 

Let $v \in \V$, then, since $v = 0$ in $\W^c$,
\begin{align*} 
\|v\|_\V^2 & = \\
& = \frac{C(n,s)}{2} \iint_{\W\times \W} \frac{|v(x)-v(y)|^2}{|x-y|^{n+2s}} dx \, dy  + C(n,s) \iint_{\W\times\W^c} \frac{|v(x)|^2}{|x-y|^{n+2s}} dx dy \\
& \leq c(n,s) \left[|v|^2_{H^s(\W)} + \int_\W |v(x)|^2 \int_{B(x, \delta(x))^c} \frac{1}{|x-y|^{n+2s}} \, dy \, dx \right] \\
& = c(n,s) \left[ |v|_{H^s(\W)}^2 + \int_\W \frac{|v(x)|^2}{\delta(x)^{2s}} \, dx \right]. 
\end{align*}
Applying the Hardy inequality \eqref{eq:hardychico}, the first estimate follows. \qquad
\end{proof}

\section{Sobolev regularity} \label{sec:regularity}
The purpose of this section is to provide regularity estimates for solutions of \eqref{eq:fraccionario} in terms of fractional Sobolev norms. We start by reviewing some key results given in \cite{RosOtonSerra}.

\begin{theorem}
[See Prop. 1.1 in \cite{RosOtonSerra}]
\label{teo:ROS}
If $\W$ is a bounded, Lipschitz domain satisfying the exterior ball condition and $f\in L^\infty(\W)$, then any solution $u$ of \eqref{eq:fraccionario} belongs to $C^s(\rn)$ and 
\begin{equation} \|u\|_{C^s(\rn)} \leq C(\W,s) \|f\|_{L^\infty(\W)}. \label{eq:holder} \end{equation}
\end{theorem}

Moreover, if $f$ is Hölder continuous, estimates for higher order Hölder norms of $u$ can be obtained. 
For $0<\beta$, we denote by $|\cdot|_{C^\beta(\W)}$ the $C^\beta(\W)$ seminorm. For $\theta \geq - \beta$, let us write $\beta = k + \beta'$ with $k$ integer and $\beta' \in (0,1]$. 
Following \cite{RosOtonSerra} we define the seminorm 
$$|w|_{\beta}^{(\theta)} = \sup_{x,y \in \W} \delta(x,y)^{\beta+\theta} \frac{|D^kw(x) - D^k w(y)|}{|x-y|^{\beta'}},$$
and the associated norm $\| \cdot \|_{\beta}^{(\theta)}$ in the following way: for $\theta \geq 0$,
$$\| w \|_\beta^{(\theta)} = \sum_{\ell=0}^k \left( \sup_{x \in \W} \delta(x)^{\ell+\theta} |D^\ell w(x)| \right)+ |w|_{\beta}^{(\theta)},$$
while for $-\beta<\theta<0$,
$$\| w \|_\beta^{(\theta)} = \| w \|_{C^{-\theta}(\W)} + \sum_{\ell=1}^k \left( \sup_{x \in \W} \delta(x)^{\ell+\theta} |D^\ell w(x)| \right) + |w|_\beta^{(\theta)} .$$
It holds,
\begin{theorem}[See Prop. 1.4 in \cite{RosOtonSerra}]
\label{teo:ROS2} Let $\W$ be a bounded domain and $\beta>0$ be such that neither $\beta$ nor $\beta+2s$ is an integer. Let $f \in C^{\beta}(\W)$ be such that $\|f\|_\beta^{(s)} < \infty$, and $u \in C^s(\rn)$ be a solution of \eqref{eq:fraccionario}. Then, $u\in C^{\beta+2s}(\W)$ and 
$$ \| u \|_{\beta + 2s }^{(-s)} \leq C(\W,s,\beta) \left( \|u\|_{C^s(\rn)} + \| f \|_{\beta}^{(s)} \right) .$$
\end{theorem} 
In the next remarks we explore some consequences of the previous theorems written in a way useful
in the sequel.
\begin{remark}[Case $0<s<\frac{1}{2}$]
 \label{rem:rossmenorunmedio}
 Taking $\beta \in (0, 1-2s)$ in Theorem \ref{teo:ROS2}, we get that there exists a constant $C(\W,s,\beta)$ such that
\begin{equation}
\sup_{x, y \in \W} \delta(x,y)^{\beta+s} \, \frac{|u(x)-u(y)|}{|x-y|^{\beta+2s}} \leq C \left( \| f \|_{L^\infty(\W)} + \| f \|_{\beta}^{(s)} \right).
\label{eq:estimacion}
\end{equation}
Moreover, since $\beta < 1$, for $f \in C^{\beta}(\W)$ it is simple to prove that
\begin{equation*}
\| f\|_{\beta}^{(s)} \leq C(\W,s) \|f\|_{C^{\beta}(\W)} .
\end{equation*}
\end{remark}

\begin{remark}[Case $\frac{1}{2}<s<1$]
 \label{rem:rossmayorunmedio}
Considering $\beta\in (0,2-2s)$, Theorem \ref{teo:ROS2} 
implies that 
$$ \sup_{x, y \in \W} \delta(x,y)^{\beta+s} \, \frac{|Du(x)-Du(y)|}{|x-y|^{\beta+2s-1}} \leq C\left(\W, s, \beta, \| f \|_{\beta}^{(s)} \right),$$
and
$$
\sup_{x\in \W} \delta(x)^{1-s} |Du(x)|  \leq C\left(\W, s, \beta, \| f \|_{\beta}^{(s)} \right).
$$
\end{remark}
In the remainder of this section we show how to use these results to bound Sobolev norms of $u$.
In order to do that it is useful to divide $\W\times\W$ into a set in which the distance between $x$ and $y$ is bounded below by $\delta(x,y)$ and a set in which $|x-y|$ is smaller than that. Roughly, for the first set, Hölder regularity of the solution is enough to control the integrand involved in fractional seminorms of $u$, as this region is away from the diagonal. As for the second one, since the weight involving $|x-y|$ is singular at $y=x$, some extra term is required in order to control its growth; this is obtained by means of Theorem \ref{teo:ROS2}. 

It is convenient to observe that, given a function $v:\W\times\W \to \R$ such that $v(x,y) = v(y,x)$ for all $x,y \in \W$, the integral of $v$ over $\W\times\W$ equals $2$ times its integral over the set
\begin{equation}
 \label{eq:a}
 A = \{ (x,y) \in \W\times\W \, \colon \, \delta(x,y) = \delta(x) \}.
\end{equation}

 We make use of the decomposition mentioned in the previous paragraph by defining 
\begin{equation}
 \label{eq:b}
B= \{ (x,y) \in A \, \colon \, |x-y| \geq \delta(x) \}.
\end{equation}

\begin{remark}
At this point, let us recall an useful identity regarding integrability of powers of the distance to the boundary function. The following holds whenever $\alpha < 1$:
\begin{equation}
\int_\W \delta(x)^{-\alpha} dx = \mathcal{O}\left( \frac{1}{1-\alpha} \right)
\label{eq:integraborde}
\end{equation}
See, for example, the proof of Lemma 2.14 in \cite{Mazya}.
\end{remark}

\subsection{Regularity in standard fractional spaces \texorpdfstring{$(0<s<1/2)$}{Lg}}
We are now ready to prove:
\begin{proposition} \label{prop:schico}
Let $0<s<\frac{1}{2}$ and $f\in C^{\frac{1}{2}-s}(\W)$. Then, for every $\eps >0$, the solution $u$ of \eqref{eq:cont} belongs to $H^{s+\frac{1}{2}-\eps}(\W)$, with
$$|u|_{H^{s+\frac{1}{2}-\eps}(\W)} \leq \frac{C(\W,s,n)}{\eps} \, \|f\|_{C^{\frac12-s}(\W)}.$$
\end{proposition}
\begin{proof}
Take $\theta \in (s,1)$ and consider the splitting of $A$ mentioned before.  Then, applying estimate \eqref{eq:holder},
\begin{align*}
\iint_{B}\frac{|u(x)-u(y)|^2}{|x-y|^{n+2\theta}} & dx \, dy \leq \\ 
& \leq C(\W,s) \|f\|_{L^\infty(\W)}^2 \int_{\W} \int_{B(x,\delta(x))^c} |x-y|^{-n-2\theta+2s} dy \, dx \\
& \leq \frac{C(\W,s,n) \|f\|_{L^\infty(\W)}^2}{\theta-s} \int_\W \delta(x)^{2(s-\theta)} dx .
\end{align*}

A necessary and sufficient condition for the finiteness of the right hand side in the previous inequality is that $\theta < s + \frac{1}{2}$.

On the other hand, assume $f \in C^\beta(\W)$ for some $\beta > 0$. In a similar fashion the application of inequality \eqref{eq:estimacion} yields
\begin{align*}
\iint_{A\setminus B} \frac{|u(x)-u(y)|^2}{|x-y|^{n+2\theta}}& dx \, dy \, \leq \\ 
\leq & \, C \int_\W \delta(x)^{-2(\beta+s)} \left(\int_{B(x,\delta(x))}  |x-y|^{-n-2\theta+2\beta+4s} dy \right) dx .
\end{align*}
Now, the integral over $B(x,\delta(x))$ is finite if and only if $\beta + 2s > \theta$. So, in this case we obtain
\begin{equation}
\iint_{A\setminus B} \frac{|u(x)-u(y)|^2}{|x-y|^{n+2\theta}} dx \, dy \, \leq C \int_\W \delta(x)^{2(s-\theta)} dx ,
\label{eq:estchico} \end{equation} 
where in the end the constant is of the form 
$$C=\frac{C(\W,s,n,\beta)}{\beta+2s-\theta} \| f \|_{C^\beta(\W)}^2 .$$
Once again, the integral in the right hand side of \eqref{eq:estchico} is finite if and only if $\theta < s + \frac{1}{2}$. If $\beta = \frac{1}{2}-s$, choosing $\theta = s + \frac{1}{2}- \eps$, we find
$$\iint_{A\setminus B} \frac{|u(x)-u(y)|^2}{|x-y|^{n+2\theta}} dx \, dy \, \leq \frac{C(\W,s,n)}{\eps} \int_\W \delta(x)^{-1+2\eps} dx .$$
Since the integral in the right hand side is $\mathcal{O}(\eps^{-1})$ (recall identity \eqref{eq:integraborde}), the proof is concluded. \qquad
\end{proof}

\begin{remark} \label{remark}
If $f$ is more regular than $C^{\frac12 -s}(\W)$, then no further gain of regularity from estimate \eqref{eq:estimacion} is possible by means of the technique of the previous proof. This is indeed sharp, see Remark \ref{rem:sharp}. The matter is that the parameter $\beta$ disappears in inequality \eqref{eq:estchico}.
\end{remark}

\subsection{Regularity in fractional weighted spaces \texorpdfstring{$(1/2<s<1)$}{Lg}}
\label{sub:fws}
Next we show that an analogous of Proposition \ref{prop:schico} is possible for $1/2<s<1$ and hence $1/2$ derivative can be gained in the a priori estimate.  Moreover, along the proof of this result it becomes clear that
the singular behavior of the solution can be localized near the boundary.
Therefore, introducing appropriate weights we find
alternative regularity results that can be used to build a priori adapted meshes. 

Before proceeding let us notice that  the expected gain of $1/2$ derivative would imply that the solution belongs at least to $H^1(\W)$. This fact can be proved studying the  behavior of the fractional seminorms $|\cdot|_{H^{1-\eps}({\W})}$ as $\eps \to 0$. 

Let us recall this useful result proved in \cite{BourgainBrezisMironescu}:

\begin{proposition} \label{prop:BBM}
Assume $v \in L^p (\W)$, $1<p<\infty$. Then, 
\begin{equation}
\lim_{\eps \to 0} \ \eps |v|^p_{W^{1-\eps,p}({\W})} = C(n,p) |v|^p_{W^{1,p}({\W})} .
\label{eq:BBM} \end{equation}
\end{proposition}

In first place, we want to prove that for the solution $u$ of \eqref{eq:cont}, the left hand side of \eqref{eq:BBM} remains bounded as $\eps \to 0$, so that it belongs to $H^1({\W})$. For that purpose, we require the following local Hölder regularity estimate, (see \cite{RosOtonSerra}, Lemma 2.9):
\begin{lemma}
If $f \in L^\infty(\W)$ and $\gamma \in (0,2s)$, then $u$ verifies
\begin{equation}
|u|_{C^\gamma (\overline{B_R (x)})} \leq C R^{s-\gamma} \| f \|_{L^\infty(\W)} \ \forall x \in \W,
\label{eq:lema29} \end{equation}
where R = $\frac{\delta(x)}{2}$ and the constant C depends only on $\W, s$ and $\gamma$, and blows up only when $\gamma \to 2s$.
\end{lemma}

The mentioned $H^1$ regularity follows from the previous results, and its proof can be obtained  like the one of Proposition \ref{prop:schico}.

\begin{lemma} \label{lema:H1}
If $1/2<s<1$ and $f\in L^\infty(\W)$, then a solution $u$ of \eqref{eq:cont} belongs to $H^1({\W})$ and therefore to $H^1(\rn)$. Moreover, it satisfies
$$|u|_{H^1(\W)} \leq \frac{C(\W,s,n) \| f\|_{L^\infty(\W)}}{(1-s)(2s-1)},$$
where the constant $C(\W,s,n)$ is uniformly bounded for all $s \in (1/2,1)$.
\end{lemma}
\begin{proof}
 Take $\eps \in (0, 1-s)$ and in the same fashion as before consider the sets $A$ and $B$, with the slight difference of a $\frac{\delta(x)}{2}$ instead of a $\delta(x)$ in the definition of the latter. Taking $\gamma = 1 - C(\eps)$ for some $0< C(\eps) < \eps$ to be chosen, applying estimate \eqref{eq:lema29} and proceeding as in the proof of Proposition \ref{prop:schico}, it follows
$$\iint_{A \setminus B} \frac{|u(x)-u(y)|^2}{|x-y|^{n+2(1-\eps)}} \, dy \, dx \leq \frac{C(\W,s,n) \|f\|^2_{L^\infty(\W)}}{\eps - C(\eps)} \int_\W \delta(x)^{2(s-1+\eps)} dx .$$
Observe that the constant $C$ in the previous inequality remains bounded for $s \in (1/2,1)$, and that the integral is $\mathcal{O}\left((2s-1+2\eps)^{-1}\right)$.

On the other hand, taking into account the global Hölder regularity of $u$ it is immediate to obtain
$$\iint_{B} \frac{|u(x)-u(y)|^2}{|x-y|^{n+2(1-\eps)}} \, dy \, dx \leq \frac{C(\W,s,n) \|f\|^2_{L^\infty(\W)}}{1-s+\eps} \int_\W \delta(x)^{2(s-1+\eps)} dx .$$
Combining the previous estimates, we get
$$|u|^2_{H^{1-\eps}(\W)} \leq \frac{C(\W,s,n) \|f\|^2_{L^\infty(\W)}}{(\eps - C(\eps))(1-s+\eps)(2s-1+\eps)},$$
where the constant $C(\W,s,n)$ remains bounded for $s \in (1/2,1)$. Taking $C(\eps)$ such that $\eps - C(\eps) = \mathcal{O}(\eps)$, the desired conclusion follows thanks to Proposition \ref{prop:BBM}. \qquad
\end{proof}

Next, we require some regularity on $Du$. Let $\beta \in (0, 2-2s)$ and assume that $f \in C^\beta (\W)$. 
Consider the subsets $A$ and $B$ of $\W \times\W$ as before (see eqs. \eqref{eq:a} and 
 \eqref{eq:b}) and introduce the weighted integral
$$
I:=\iint_{A \setminus B} \frac{|Du(x)-Du(y)|^2}{|x-y|^{n+2(\ell-1)}} \, \delta(x,y)^{2\alpha} dx \, dy.
$$
Using the first inequality of Remark \ref{rem:rossmayorunmedio} we explore how  to take
the involved parameters $\ell$ and $\alpha$ in order to keep $I$ bounded.  

$$ I\leq C \int_\W \left(\int_{B(x,\delta(x))} |x-y|^{2(\beta + 2s - 1)-n-2(\ell-1)} \,dy \right) \delta(x)^{2(\alpha - \beta - s)} dx \leq $$
$$\leq \frac{C}{\beta + \ell - 2s} \int_\W \delta(x)^{2(\alpha + s -\ell)} dx \leq \frac{C}{(\beta + \ell - 2s)(1+2(\alpha - s - \ell))} ,$$  
where, in order to ensure the convergence of the integrals involved, we must require
\begin{equation}
\label{eq:cotasReg}
\ell - \beta < 2s \ \mbox{ and } \ \ell < \alpha + s + 1/2.
\end{equation}
On the other hand, for
$$
II:=\iint_{B} \frac{|Du(x)-Du(y)|^2}{|x-y|^{n+2(\ell-1)}} \, \delta(x,y)^{2\alpha} dx \, dy,
$$
again due to Remark \ref{rem:rossmayorunmedio},
$$II\leq C \int_\W \left(\int_{B(x,\delta(x))^c} |x-y|^{-n-2(\ell-1)} dy \right) \delta(x)^{2(\alpha + s - 1)} dx \leq $$
$$\leq C \int_\W \delta(x)^{2(\alpha + s -\ell)} dx \leq \frac{C}{1+2(\alpha - s - \ell)} ,$$
where the condition for the finiteness of $II$ is guaranteed 
if we restrict our attention to \eqref{eq:cotasReg}.
Under these conditions, we have proved:
\begin{equation}\label{eq:conysinpeso}
|Du|_{H^{\ell}_\alpha (\W)} \leq \frac{C}{(\beta + \ell - 2s)(1+2(\alpha - s - \ell))} .
\end{equation}

Within the range provided in  \eqref{eq:cotasReg} we can highlight
some cases of interest. In the same spirit of Proposition \ref{prop:schico}, we have, considering $\alpha=0$ and $\ell=s+1/2-\varepsilon$ in \eqref{eq:conysinpeso}:

\begin{proposition} \label{proposition:adentro}
If $1/2<s<1$ and  $f\in C^{\beta}(\W)$ for some $\beta>0$, then the solution $u$ of \eqref{eq:cont} belongs to $H^{s+\frac{1}{2}-\eps}(\W)$ for all $\eps >0$, with
$$ |Du|_{H^{s-\frac{1}{2}-\eps}(\W)} \leq \frac{C(\W,s,n,\beta)}{\sqrt{\eps} (2s-1)} \,  \| f\|_{C^\beta(\W)} .$$
\end{proposition}

If we restrict the weight to the Muckenhoupt $A_2$ class (see Remark \ref{rem:pesosA2}), which can be relevant for extending this considerations to the global case treated later, we need to choose $\alpha<1/2$. This restriction is 
also of importance in the optimality of the graded meshes proposed later. Accordingly, assume $\alpha = 1/2 - \eps$ for  $\eps>0$ small enough and take $\ell = 1 + s - 2\eps$ and $\beta = 1 -s$.  From \eqref{eq:conysinpeso} we obtain the following weighted version, where the gaining of regularity is of almost one derivative:
\begin{proposition}
\label{prop:sgrandepesos}
Let $1/2<s<1$, $f \in C^{1-s}(\W)$ and $u$ be the solution of our problem. Then, given $\eps>0$ it holds that $u \in H^{1+s-2\eps}_{1/2-\eps}(\W)$ and
$$\|u\|_{H^{1+s-2\eps}_{1/2-\eps}(\W)} \leq \frac{C(\W, s, \| f \|_{1-s})}{\eps} .$$
\end{proposition}
\begin{remark}
\label{rem:sharp}
The regularity estimates given in this section 
are sharp, in the sense that if we consider the problem
\begin{equation}
\left\lbrace
  \begin{array}{l}
      (-\Delta)^s u = 1  \mbox{ in } B(x_0,r), \\
      u = 0 \mbox{ in }B(0,r)^c , \\
      \end{array}
    \right.
\label{eq:ejemplo} \end{equation}
for $x_0 \in \rn$ and $r>0$, then its solution is given by \cite{Getoor}
$$u(x) = \frac{2^{-2s}\Gamma\left(\frac{n}{2}\right)}{\Gamma\left(\frac{n+2s}{2}\right)\Gamma\left(1+s\right)} \left(r^2 - |x-x_0|^2 \right)^s \ \mbox{in } B(x_0, r).$$
It is straightforward to check that this function belongs to $H^{s+\frac{1}{2}-\eps}(\W)$ for all $s \in (0,1)$, to ${H^{1+s-2\eps}_{1/2-\eps}(\W)}$ if $s\in (1/2,1)$ and that the parameter
$\eps$ can not be removed.
\end{remark}

\subsection{Global Regularity}
A direct derivation of global regularity is a simple task in the present context. First we present the following lemma.
\begin{lemma}\label{lema:afuera} For $\frac{1}{2}<s<1$, $\eps>0$ and $u \in H^{s+\frac{1}{2} -\eps}(\W)$, it holds
$$\int_\W \int_{\W^c} \frac{|Du(x)|^2}{|x-y|^{n+2(s-\frac{1}{2}-\eps)}} \, dy \, dx \leq \frac{C(\W,s,n)}{2s-1-2\eps} \|Du\|_{H^{s-\frac{1}{2}-\eps}(\W)}	^2 . $$
\end{lemma}
\begin{proof}
This is a simple consequence of the inclusion $\W^c \subset B(x,\delta(x))^c$ for all $x \in \W$ and the Hardy inequality \eqref{eq:hardychico}. \qquad
\end{proof}

Combining Lemmas \ref{lema:H1},  \ref{lema:afuera}, and Proposition
\ref{proposition:adentro} we have proved:

\begin{proposition}\label{prop:sgrande}
If $1/2<s<1$ and  $f\in C^\beta(\W)$ for some $\beta>0$, then the solution $u$ of \eqref{eq:cont} belongs to $H^{s+\frac{1}{2}-\eps}(\rn)$ for all $\eps >0$ and
$$ | u |_{H^{s+\frac{1}{2}-\eps}(\rn)} \leq \frac{C(\W,s,n,\beta)}{\sqrt{\eps}(2s-1)} \, \|f\|_{C^{\beta}(\W)}. $$
\end{proposition}
In a similar fashion we get
\begin{proposition}
Let $1/2<s<1$, $f \in C^{1-s}(\W)$ and $u$ be the solution of our problem. Then, given $\eps>0$, $u \in H^{1+s-2\eps}_{1/2-\eps}(\rn)$ and
$$\|u\|_{H^{1+s-2\eps}_{1/2-\eps}(\rn)} \leq \frac{C(\W, s, \| f \|_{1-s})}{\eps} .$$
\end{proposition}

\subsection{The case \texorpdfstring{$s=1/2$}{Lg}} Up to now, the possibility of $s$ being equal to $1/2$ has been excluded from our analysis. In order to obtain a regularity estimate, the arguments to be carried are in the same spirit as before; the only issue to overcome is the need for $\beta>0$ in Theorem \ref{teo:ROS2}. In this case, the argument demands less regularity of the function $f$. Indeed, the same technique as in the proof of Lemma \ref{lema:H1}  gives $u \in H^{1-\eps}(\W)$ for all $\eps >0$, with a bound of the type
\begin{equation}
|u|_{H^{1-\eps}(\W)} \leq \frac{C(\W,n) }{\eps } \, \| f \|_{L^\infty(\W)} .
\label{eq:reg12} \end{equation}
Observe that we cannot assure $u \in H^1(\W)$ by taking $\eps \to 0$ in the previous inequality, which is coherent with example \eqref{eq:ejemplo}.

Moreover, in this case the space $\V$ coincides with the Lions-Magenes space $H^{1/2}_{00} (\W)$, which is strictly contained in  $H^{1/2}(\W)$  (\cite{LionsMagenes}, Theorem 1.11.7). Actually, the energy norm
is equivalent to $\|u\|_{H^{1/2}(\W)} + \| \delta^{-1/2} u \|_{L^2(\W)}$. By means of the Hardy inequality it can be bounded by $|u|_{H^{1/2+\eps}(\W)}$ for any $\eps>0$. As a consequence, FE error estimates for this case follow from the theory developed below for $s\neq 1/2$.

%%%%%%%%%%%%%%%%%%%%%%%%%%%%%%%%%%%%%%%%%%%%%%%%%%%%%%%%%%%%%%%%%%%%%%%%%%%%%%%%%%%%%%%%%%%%%%%%%%%%%%%%%%%%%%%%%%%%%

\section{Finite Element approximations} \label{ss:fem} 
We restrict the analysis to FE approximations of equation \eqref{eq:cont} to piecewise linear functions. The simpler case of $\mathcal{P}_0$,
which  provides a  conforming method for $s<1/2$,
is not addressed here in order to  present an unified approach for the whole range $0<s<1$.  

We assume that $\cup_{T\in \mathcal{T}_h}T=\bar\Omega$ where $\mathcal{T}_h$ is an admissible  triangulation of $\W$, made up of elements $T$ of diameter $h_T$ and with $\rho_T$ equal to the diameter of the largest ball contained in $T$.

We require that the family of triangulations under consideration satisfies:
\begin{align}
& \exists \sigma > 0 \mbox{ s.t. } h_T \leq \sigma \rho_T \ \forall T \in \mathcal{T}_h , \tag{Regularity} \label{eq:regularity} \\
& \exists \lambda > 0 \mbox{ s.t. } h_T \leq \lambda h_{T'} \ \forall T , T' \in \mathcal{T}_h \colon \bar{T} \cap \bar{T'} \neq \emptyset. 
 \tag{Local quasi-uniformity} \label{eq:uniformity}
\end{align}
Naturally the second condition is a consequence of the first one. In this way $\lambda$
can be expressed in terms of $\sigma$. 

Consider the discrete space
$$\V_h = \{ v \in \V \, \colon \, v \big|_T \in \mathcal{P}_1 \, \forall T \in \T_h \} .$$
It is immediate to check that there exists a unique solution to the discrete problem
\begin{equation}
\mbox{find $u_h \in \V_h$ such that } a(u_h, v_h) = \int_\W f v_h \ \forall v_h \in \V_h,
\label{eq:disc} \end{equation}
where $a$ is the bilinear form defined by \eqref{eq:bilineal} and that C\'ea's Lemma holds in this context. Namely, the FE solution is the best approximation in $\V_h$ to the solution of problem \eqref{eq:fraccionario}:
\begin{equation}
\| u - u_h \|_\V = \min_{v_h \in \V_h} \| u - v_h \|_\V .
\label{eq:cea}
\end{equation}

It seems clear that the norm $\| \cdot \|_\V$ is the `natural one' for studying our problem. A simple trick (see Lemma \ref{lemma:truco}) shows indeed that although it involves an integration over an unbounded domain, it is possible to carry the computation of the error under this norm by integrating in $\W$.

%%%%%%%%%%%%%%%%%%%%%%%%%%%%%%%%%%%%%%%%%%%%%%%%%%%%%%%%%%%%%%%%%%%%%%%%%%%%%%%%%%%%%%%%%%%%%%%%%%%%%%%%%%%%%%%%%%%%%%

\subsection{Estimates for the Scott-Zhang interpolation operator} \label{sec:interpolation}
The next step towards obtaining a finite element error estimate is to find an adequate interpolation or projection operator. One difficult aspect dealing with fractional seminorms is that they are not additive with respect to the decomposition of domains. Nevertheless some localization is possible \cite{Faermann2, Faermann}.

\begin{lemma}\label{faermann} Let $s \in (0,1)$ and $\Omega$ a Lispchitz domain. Then, for any $v \in H^s(\W)$ it holds that
$$|v|_{H^s(\W)}^2 \leq C(n,s) \sum_{T \in \T_h} \left[ \int_T \int_{S_T} \frac{|v (x) - v (y)|^2}{|x-y|^{n+2s}} \, dy \, dx + h_T^{-2s} \| v \|^2_{L^2(T)} \right],$$
where
$$S_T := \bigcup_{T' \colon \bar{T'}\cap \bar{T} \neq \emptyset} T'.$$

\end{lemma}

In the following, let $\Pi_h v$ denote the Scott-Zhang interpolator of $v$. Let us recall its basic properties \cite{ScottZhang}.

\begin{theorem} Let $\ell > 1/2$, then $\Pi_h : H^{\ell}(\W) \to \V_h$ satisfies that $\Pi_h (v_h) = v_h$ for all $v_h \in \V_h$ and $\Pi_h$ preserves boundary conditions, in the sense that $H^{\ell}_0 (\W)$ is mapped to 
$\V_{h \, 0} := \{ v_h \in \V_h \, \colon \, v_h \big|_{\pp \W} = 0 \} $.
\end{theorem}

Stability and approximability results for the Scott-Zhang interpolation in fractional spaces  were studied in \cite{Ciarlet}, where the following result is proved:

\begin{theorem} \label{teo:ciarlet} Given $\frac{1}{2} < \ell < 1$ and $0\leq t \leq \ell$, then $\forall T \in \T_h, \ v \in H^\ell (S_T)$
\begin{equation}
| \Pi_h v |_{H^t(T)} \leq C(n,s,\sigma) \left(h_T^{-t} \| v \|_{L^2(S_T)} + h_T^{\ell-t} |v|_{H^\ell(S_T)}\right),
\label{eq:estciarlet} \end{equation} 
and 
\begin{equation}
\| v -  \Pi_h v \|_{H^t(T)} \leq C(n,s,\sigma) h_T^{\ell-t} |v|_{H^\ell(S_T)} \quad \forall v \in H^\ell (S_T). \label{eq:aproxciarlet}
\end{equation} 
\end{theorem}

The procedure towards obtaining the previous results is as follows: the stability of the operator \eqref{eq:estciarlet} relies on basic estimates for the parametrization of $T$ and of the basis functions involved;  in order to obtain \eqref{eq:aproxciarlet}, it is enough to apply the stability estimate, the Bramble-Hilbert lemma and interpolate between some $L^2(T)$ and $H^1(T)$ estimates. 

The same machinery as in \cite{Ciarlet} and the previous lemma yields the following stability type estimate.

\begin{proposition} \label{prop:casi}
Let $T \in \T_h$
\begin{enumerate}
 \item If $s \in (0,1/2],  \, \ell \in (1/2,1)$ and  $\ v \in H^\ell (\W)$
\begin{align*}
& \int_T \int_{S_T} \frac{|\Pi_h v (x) - \Pi_h v (y)|^2}{|x-y|^{n+2s}} \, dy \, dx \ \leq  \\
&\leq C(n,s,\sigma) \left[h_T^{-2s} \| v \|^2_{L^2(S_T)} + h_T^{2\ell-2s} | v |^2_{H^\ell(S_T)} \right].
\end{align*}
\item If $s\in (1/2,1)$ and $\ v \in H^1 (\W)$
\begin{align*}
\int_T \int_{S_T} & \frac{|\Pi_h v (x) - \Pi_h v (y)|^2}{|x-y|^{n+2s}} \, dy \, dx \ \leq  \\
&\leq C(n,s,\sigma) \left[h_T^{-2s} \| v \|^2_{L^2(S_T)} + h_T^{2-2s} |v|^2_{H^1(S_T)}  \right]. 
\end{align*}
\end{enumerate}
\end{proposition}
\begin{remark}
Regarding the behavior of the constant $C$  in  
the variable $s$, it is easy to check that $C\sim \frac{1}{1-s}$. 
\end{remark}

Before obtaining error estimates for $\Pi_h u$
we recall some facts about a well known
key tool.  Let  $S$ be an star-shaped
domain w.r.t. a ball $B$.  Introduce the  polynomial $\mathcal{P}_k(u)$ of degree $k$ with the property
$$
\int_{S}D^\alpha\left(u-\mathcal{P}_k(u)\right)=0,
$$
for $0\le |\alpha|\le k$. In our context we 
need to focus on the cases $k=0,1$. For instance,
Proposition \ref{prop:poincareII} gives at once
$$
\|u-\mathcal{P}_0(u)\|_{L^2(S_T)}\le Ch^\ell|u|_{H^\ell(S_T)},
$$
for $0<\ell<1$ and with a constant depending on the chunkiness parameter of $S_T$.
In this context we can write 
$C=C(\sigma)$ (thanks to the mesh properties \eqref{eq:regularity} and \eqref{eq:uniformity}).

Since $|u-\mathcal{P}_0(u)|_{H^\ell(S_T)}=|u|_{H^\ell(S_T)}$, by means of the $L^2$ estimate and interpolation of spaces we obtain
$$
|u-\mathcal{P}_0(u)|_{H^s(S_T)}\le Ch^{\ell-s}|u|_{H^\ell(S_T)},
$$
for any  $0<s<\ell<1$, with a constant $C=C(\sigma)$.

Similarly, using  the standard Poincar\'e inequality for functions with zero average together with Proposition \ref{prop:poincareII} , we obtain  for any $1<\ell<2$
\begin{equation}
 \label{eq:estP}
 \|u-\mathcal{P}_1(u)\|_{L^2(S_T)}+h_T|u-\mathcal{P}_1(u)|_{H^1(S_T)} \le Ch_T^{\ell}|u|_{H^\ell(S_T)},
\end{equation}
with $C$ uniformly bounded in terms of  $\sigma$.

Moreover, interpolation of spaces 
and  \eqref{eq:estP} give for $0<s<1$ and $1<\ell<2$
\begin{equation}
\label{eq:estPs}
|u-\mathcal{P}_1(u)|_{H^s(S_T)}\le Ch_T^{\ell-s}|u|_{H^\ell(S_T)},
\end{equation}
with $C$ bounded again in terms of  $\sigma$.

\subsection{Uniform Meshes}
\label{sub:unif}
From (1) (resp. (2)) of Proposition \ref{prop:casi} and approximation properties of  $\mathcal{P}_0$ (resp. $\mathcal{P}_1$) for $0<s<\ell<1$ (resp. $1/2<s<1$ and $1<\ell<2$) it follows, in an standard fashion, the local approximability inequality
$$\int_T \int_{S_T} \frac{|(v-\Pi_h v) (x) - (v-\Pi_h v) (y)|^2}{|x-y|^{n+2s}} \, dy \, dx \leq C(n,s,\sigma) h_T^{2\ell-2s} |v|_{H^\ell(S_T)}^2.$$
Therefore, 
$$\| v - \Pi_h v \|_{\V} \leq C(n,s,\sigma)  h^{\ell- s} |v|_{H^\ell(\W)}. $$

Calling $h = \max_{T \in \mathcal{T}_h} h_T$, the mesh size parameter, invoking identity \eqref{eq:cea}, and combining this respectively with Proposition \ref{prop:schico}, estimate \eqref{eq:reg12} and Proposition \ref{proposition:adentro}, we have proved:
\begin{theorem}\label{teo:esterror}
For the solution $u$ of \eqref{eq:cont} and its FE approximation $u_h$ given by \eqref{eq:disc} we have the a priori estimates
\begin{align*}
\| u - u_h \|_\V & \leq  \frac{C(s,\sigma)}{\eps} h^{\frac{1}{2}-\eps} \|f\|_{C^{\frac12 - s}(\W)} \quad \forall \eps>0, \quad \mbox{if } s < 1/2, \\
\| u - u_h \|_\V &\leq  \frac{C(\sigma)}{\eps} h^{\frac{1}{2}-\eps} \|f\|_{L^\infty(\W)} \quad  \forall \eps>0, \quad \mbox{if } s = 1/2, \\
\| u - u_h \|_\V &\leq  \frac{C(s,\beta,\sigma)}{\sqrt{\eps} (2s-1)} h^{\frac{1}{2}-\eps} \|f\|_{C^{\beta}(\W)} \quad \forall \eps>0, \quad \mbox{if } s > 1/2 .
\end{align*}
So, if $h$ is sufficiently small, taking $\eps = |\ln{h}|^{-1}$ we obtain the quasi-optimal estimates
\begin{align*}
\| u - u_h \|_\V & \leq  C(s,\sigma) h^\frac12 |\ln h|  \|f\|_{C^{\frac12-s}(\W)} , \mbox{ if } s<1/2, \\
\| u - u_h \|_\V & \leq  C(\sigma) h^\frac12 |\ln h|  \|f\|_{L^\infty(\W)} , \mbox{ if } s = 1/2,  \\
\| u - u_h \|_\V & \leq  \frac{C(s,\beta,\sigma)}{2s-1} h^\frac12 \sqrt{|\ln h|}  \|f\|_{C^\beta(\W)} , \mbox{ if } s>1/2. 
\end{align*} 
\end{theorem}

\subsection{Graded Meshes}
The obtained approximability property of $\Pi_h$  is enough to deal with standard fractional spaces. Nevertheless, for $1/2<s<1$ it is possible to improve the convergence rate using graded meshes. It requires dealing with the weights already introduced in Subsection \ref{sub:fws}. In order to get appropriate bounds in these spaces we should replace the classical Poincar\'e inequality of Proposition \ref{prop:poincareII} by an improved version in fractional spaces. The term  \emph{improved} in this context usually involves weights which are
powers of the distance to the boundary. On the other hand, in \cite[Theorem 3.1]{Ritva} we find  the following fractional improved Sobolev-Poincar\'e inequality for functions with zero average
\begin{equation}
\left( \int_\W |v(x)|^q dx \right)^{\frac{1}{q}} \leq C \left(\int_\W \int_{\W \cap B(x, \tau \delta(x))} \frac{|v(x)-v(y)|^p}{|x-y|^{n+\sigma p} } \, dy \, dx \right)^\frac{1}{p}.
\label{eq:improved}
\end{equation}
The parameters $\tau$ and $\sigma$ can be taken arbitrarily as long as  $\tau,\sigma \in (0,1)$, while  $1 < p \leq q \leq \frac{np}{n-\sigma p}$, $p < n/\sigma$. 
In \cite{Ritva} the domain $\W$ is assumed to belong to the class of John domains (for a definition and properties of this class see for instance \cite{Martio}); this class is much broader than the one of star-shaped domains.

Now we set $\tau = 1/2$ in \eqref{eq:improved}. For $\sigma$ to be chosen, we consider $p$ such that $\frac{np}{n-\sigma p} = 2 = q$. Observe that this election obviously implies that $p<2$, and therefore for all $\alpha \in \mathbb{R}$, applying H\"older's inequality with exponents $\frac2p$ and $\frac{2}{2-p}$,
$$\| v \|_{L^2(\W)} \leq C I_1^\frac{1}{2}I_2^\frac{2-p}{2p},$$
where
 $$I_1= \int_\W \int_{\W \cap B(x, \frac{\delta(x)}{2})} \frac{|v(x)-v(y)|^2}{|x-y|^{n+2s} } \, \delta(x,y)^\frac{2 \alpha}{p} dy \, dx ,$$
 and
$$ I_2=\int_\W \int_{\W \cap B(x, \frac{\delta(x)}{2})} |x-y|^{-n+ \frac{2p (s-\sigma)}{2-p}} \delta(x,y)^{-\frac{2\alpha}{2-p}} \, dy \, dx. $$
Since for every $x \in \W$ and $y \in B(x, \frac{\delta(x)}{2})$ it holds that $\delta(x,y) \in \left[\frac{\delta(x)}{2}, \delta(x) \right]$, assuming that $\sigma < s$ the second integral $I_2$ can be estimated as follows:

$$I_2 \leq \\
C \int_\W \left( \int_0^{\frac{\delta(x)}{2}} \rho^{-1 + \frac{2p (s-\sigma)}{2-p}} d\rho \right) \delta(x)^{-\frac{2\alpha}{2-p}} dx  
 \leq C \int_\W \delta(x)^{\frac{2p (s-\sigma) - 2\alpha }{2-p}} dx. $$

This is finite if and only if $\frac{2p (s-\sigma) - 2\alpha }{2-p} > -1$, and recalling the election of $p$ we made, it is enough to consider 
$$\alpha < \frac{2n(s-\sigma)+2\sigma}{n+2\sigma} .$$
Choosing $\alpha$ according to this restriction, we obtain that the weight in the term $I_1$ must be
$$ \frac{2 \alpha}{p} < 2 s - 2 \sigma \left( 1 - \frac{1}{n} \right) .$$
Therefore, taking $\sigma = \frac{\eps n}{2 (n-1)}$ for some $\eps > 0$ small enough, we obtain the Proposition \ref{prop:improved1} below. The result is stated here for an star-shaped domain $S$.  Actually, the constant $C$ in \eqref{eq:improved} depends on the constants associated to the John domain $\Omega$. In the case of a star-shaped domain the John constants are easily bounded in terms of the chunkiness parameter. Working with a domain of diameter one, a further scaling argument shows the final dependence on the diameter of $S$.
\begin{proposition}[Weighted fractional Poincar\'e inequality] 
\label{prop:improved1} Let $0<s<1$, $\alpha < s$
and $S$ a domain which is star-shaped w.r.t. a ball $B$. Then, there exists a constant $C$ such that for every $v \in L^2 (S)$ with $\int_S v = 0$, it holds
\begin{equation}
\| v \|_{L^2(S)} \leq C d_S^{s-\alpha}|v|_{H^s_{\alpha}(S)},
\label{eq:weightedpoincare}
\end{equation}
with a constant $C$ depending on the chunkiness parameter. 
\end{proposition}

\begin{remark}
The previous result as stated suffices to fulfill our needs. Nevertheless, from the 
standard version with $s=1$ one might expect \eqref{eq:weightedpoincare} to hold even if $\alpha=s$. This is indeed the case, however we were unable to produce a proof as short as the one given here for $\alpha<s$. 
\end{remark}

Now we want to exploit the previous proposition together with
the a priori estimate of Proposition \ref{prop:sgrandepesos}. Since the weights under consideration vanish only on the boundary of the domain we 
need to rely on \eqref{eq:weightedpoincare} just  for patches $S_T$ touching $\partial\W$. Actually, for them we obtain the following improved version of \eqref{eq:estP}, derived using Proposition \ref{prop:improved1} instead of Proposition \ref{prop:poincareII}
\begin{equation*}
 \|u-\mathcal{P}_1(u)\|_{L^2(S_T)}+h_T|u-\mathcal{P}_1(u)|_{H^1(S_T)} \le Ch_T^{\ell-\alpha}|u|_{H_\alpha^\ell(S_T)},
\end{equation*}
where $1<\ell<2$ and $\alpha<\ell-1$. Taking
$1/2<s<1$, $\ell=1+s-2\eps$, and $\alpha=1/2-\varepsilon$
we obtain the analogous of \eqref{eq:estPs}
\begin{equation*}
|u-\mathcal{P}_1(u)|_{H^s(S_T)}\le Ch_T^{1/2-\varepsilon}|u|_{H^{1+s-2\eps}_{1/2-\varepsilon}(S_T)}.
\end{equation*}

In particular, this property of $\mathcal{P}_1$
and the stability estimates  (see (b) in Proposition \ref{prop:casi}) yield
$$\int_T \int_{S_T} \frac{|(v-\Pi_h v) (x) - (v-\Pi_h v) (y)|^2}{|x-y|^{n+2s}} \, dy \, dx \leq C(n,s,\sigma) h_T^{1-2\varepsilon} |v|_{H^{1+s-2\eps}_{1/2-\varepsilon}(S_T)}^2.$$ 

This approximability property is particularly useful for patches $S_T$ touching the boundary of $\W$. For these it must be recalled that 
$dist(x,\partial S_T)\le dist(x,\partial \W)$ for $x\in S_T$. 

The following is standard  (see \cite[Section 8.4]{Grisvard}). We assume that, in addition to \eqref{eq:regularity} and \eqref{eq:uniformity}
our meshes enjoy some extra properties, denoted below with $(H)$. First, we pick  an arbitrary mesh size parameter $0<h$ and 
 define, for $\varepsilon$ small enough, a number $1\le \mu=2/(1+2\varepsilon)<2$. 
 
Property $(H)$:  assume that
for any $T\in \mathcal{T}_h$
\begin{itemize}
 \item If $T\cap \partial \W\neq \emptyset$,
 then $h_T\le C h^{\mu}$
 \item Otherwise $h_T\le C h \,dist(T,\partial \Omega)^{(\mu-1)/\mu} $
\end{itemize}
Using the estimates for $\Pi_h$ given in Subsection \ref{sub:unif} when  $S_T\cap \partial \W=\emptyset$ and the a priori estimate Proposition \ref{prop:sgrandepesos}, we can conclude, for graded meshes obeying $(H)$ and $1/2<s<1$, that 
$$
\| u - u_h \|_\V  \leq  \frac{C(s,\beta,\sigma)}{2s-1} h \sqrt{|\ln h|}  \|f\|_{C^{1-s}(\W)}. 
$$
If the mesh parameter $h$ can be appropriately related to the number $N$ of nodes of the mesh then it is possible to obtain quasi-optimal order of convergence.

\begin{theorem}\label{teo:esterrorgrad}
Let $1/2<s<1$ and assume that the FE triangulation $\mathcal{T}_h$ satisfies conditions \eqref{eq:regularity}, \eqref{eq:uniformity} as well as the grading 
hypotheses $(H)$. If the mesh parameter $h$ behaves like $h\sim \frac{1}{N^{1/n}}$, $N$ being the number of mesh nodes, then for the solution $u$ of \eqref{eq:cont} and its FE approximation $u_h$ given by \eqref{eq:disc} we have the following a quasi-optimal a priori estimate
 $$
\| u - u_h \|_\V  \leq  \frac{C(s,\beta,\sigma)}{2s-1} N^{-1/n} \sqrt{|\ln N|}  \|f\|_{C^{1-s}(\W)}. 
$$
\end{theorem}

\begin{remark}
 In the next section we show a concrete 2D example in which meshes of the kind required in previous theorem are constructed.
\end{remark}

%%%%%%%%%%%%%%%%%%%%%%%%%%%%%%%%%%%%%%%%%%%%%%%%%%%%%%%%%%%%%%%%%%%%%%%%%%%%%%%%%%%%%%%%%%%%%%%%%%%%%%%%%%%%%%%%%%%%%%%%%%%%%
\section{Implementation details and results}\label{sec:implementation}
Numerical computation of solutions of \eqref{eq:fraccionario} has as main difficulties the fact that a singular kernel is involved, and that integrals over the whole $\rn$ must be calculated.  

Now we will comment some features of the implementation, more details can be 
found in \cite{ABB}. Let $\{ \phii_i \}$ be the nodal basis of $\V_h$. Due to the linearity of the fractional Laplacian, we just need to solve a system $KU = F$, where the right hand side vector $F = (f_i)$ is assembled straightforwardly because
$$f_i = \int_\W \phii_i(x) f(x) \, dx . $$
The challenging task is to  accurately compute the stiffness matrix $K=(K_{ij})$, given by
$$K_{ij} = \iint_Q \frac{(\phii_i(x)-\phii_i(y))(\phii_j(x)-\phii_j(y))}{|x-y|^{n+2s}} dx \, dy .$$
Splitting $Q = (\W \times \W) \cup (\W \times \W^c) \cup (\W^c \times \W)$ and taking into account that the interactions in $\W \times \W^c$ and $\W^c \times \W$ are symmetric respect to $x$ and $y$, we get
\begin{equation*}\begin{split}
K_{ij} = & \iint_{\W \times \W} \frac{(\phii_i(x)-\phii_i(y))(\phii_j(x)-\phii_j(y))}{|x-y|^{n+2s}} dx \, dy \, + \\ 
& + 2 \iint_{\W \times \W^c} \frac{\phii_i(x)\phii_j(x)}{|x-y|^{n+2s}} dx \, dy. 
\end{split}\end{equation*}

By making a double loop over the elements of the triangulation, the integrals above can be computed. The quadrature rules employed for computing the integrals over two  elements $T$ and $T'$ (with the possibility that $T = T'$) are analogous to the ones presented in Chapter 5 of \cite{SauterSchwab}. The advantage of applying the transformations presented in that book for this problem is that they convert an integral over the product of two elements into an integral over $[0,1]^4$, in which variables can be separated and the singular part can be solved analytically. The integral involving $\W \times \W^c$ is computed resorting to integration in polar coordinates, taking into account that it will be nonzero only if $\mbox{supp} (\phii_i) \cap \mbox{supp} (\phii_j) \neq \emptyset $.

\subsection{Numerical Results for Uniform Meshes} In first place, numerical solutions of problem \eqref{eq:ejemplo} were obtained for $n=2$, $x_0=0$ and $r=1$, and for several values of $s$. The computation of the error in the $\V$ norm is easily achieved by using the following. 
\begin{lemma} \label{lemma:truco} It holds
$$ \| u - u_h \|_\V = \left(\int_\W f(x) (u(x)-u_h(x)) \ dx\right)^\frac{1}{2}. $$
\end{lemma}
\begin{proof}
It is an immediate consequence of the orthogonality condition 
$$a(v_h,u-u_h) = 0 \quad \forall v_h \in \V_h .$$
Indeed, from it we obtain
$$\| u - u_h \|_\V^2 = a(u-u_h,u-u_h) = a(u,u-u_h),$$
and the equality follows by \eqref{eq:cont} and \eqref{eq:bilineal}. \qquad
\end{proof}

Although the computation involved in this lemma is subtle in general, in this particular case it
can be carried out exactly since $f \equiv 1$ on $\W$ and a closed formula for $\int_\W u$ is easy to get while the exact value of $\int_\W u_h$ can be numerically evaluated.

Several orders are shown in Table \ref{tab:ejemplo}; these results are in accordance with the estimates in Theorem \ref{teo:esterror}. In Figure \ref{fig:ejemplo} computational errors for $s=0.5$ and $s=0.7$ are shown.
 
\begin{table}[htbp]\footnotesize\centering
		\begin{tabular}{|c| c|} \hline
		Value of $s$ & Order (in $h$) \\ \hline
		0.1 & 0.497 \\
		0.2 & 0.496 \\
		0.3 & 0.498 \\
		0.4 & 0.500 \\
		0.5 & 0.501 \\
		0.6 & 0.505 \\
		0.7 & 0.504 \\
		0.8 & 0.503 \\
		0.9 & 0.532 \\
 \hline	 
\end{tabular}\label{tab:ejemplo}
\caption{(Uniform Meshes) Computational rates of convergence for problem \eqref{eq:ejemplo}, measured in the norm $\| \cdot \|_\V$. The mesh parameter is the actual size of the elements.}
\end{table}

\begin{figure}[ht]
	\centering
	\includegraphics[width=0.75\textwidth]{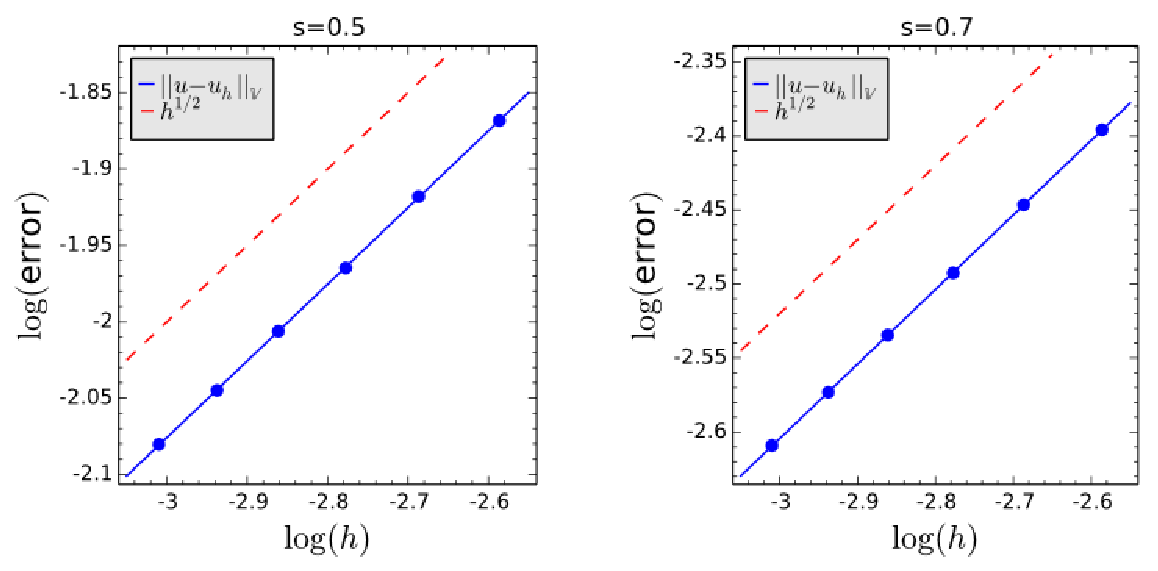}
	\label{fig:ejemplo}
	\caption{Computational results for problem \eqref{eq:ejemplo} using uniform meshes. The left panel shows the rate for $s = 0.5$ and the right one for $s = 0.7$. In both cases, the rate is $\approx 0.5$, as predicted by Theorem \ref{teo:esterror}.}
\end{figure}

As a second example, take $s > 1/2$ and consider problem \eqref{eq:fraccionario} posed on the interval $\W = (-1,1)$, with exact solution $u(x) = \sin(\pi x) \chi_{(-1,1)}(x) $, namely:
\begin{equation}
 \left\lbrace
  \begin{array}{c l}
      (-\Delta)^s u = (-\Delta)^s \sin(\pi x) & \mbox{ in } (-1,1) \\
      u = 0 & \mbox{ in } (-\infty,-1)\cup(1,\infty) .\\
      \end{array}
    \right.
\label{eq:ejemplo2}
\end{equation}
Since the solution for this problem is smooth in $(-1,1)$, the convergence in the energy norm would be expected to be of order $2-s$. Some results are shown in Table \ref{tab:ejemplo2}, where it can be seen that these orders are indeed achieved.

 \begin{table}[ht] \footnotesize\centering
\begin{tabular}{|c| c|} \hline
		Value of $s$ & Order (in $h$) \\ \hline
		%0.5 & 0. \\
		0.6 & 1.4028 \\
		0.7 & 1.2993 \\
		0.8 & 1.2002 \\
		0.9 & 1.1002 \\
 \hline	 
\end{tabular}
\caption{Rates of convergence for uniform meshes in the norm $\| \cdot \|_\V$ for problem \eqref{eq:ejemplo2} and $s > 1/2$.}
\label{tab:ejemplo2}
\end{table}

\subsection{Numerical Results for Graded Meshes}
For the same $2D$ problem of the first example we show how to build appropriate graded meshes.
Our domain $\W$ is the unitary disk. Therefore, we may pick a positive integer $M$ and define an increasing sequence of radii $r_i:=1-\left(1-\frac{i}{M} \right)^{\mu}$ for $1\le i \le M$. We can mesh the complete disk $\Omega$ by meshing each subdomain $\W_i=\{x\in \W:r_{i-1}<|x|< r_i\}$ with uniform elements of size $h_T=h_i=r_i-r_{i-1}$ (see Figure \ref{fig:ejemplograd}). proceeding in that way it is possible to compute the final number of nodes $N\sim \sum_{i=1}^{M}1/h_i$. It is a simple exercise to check that \emph{if} $\mu<2$ then $N\sim M^2$.
The previous construction ensures that  conditions \eqref{eq:regularity}, \eqref{eq:uniformity} 
and hypotheses $(H)$ hold, taking  $h=1/M$. 

Table \ref{tab:ejemplo3} shows numerical results for this case. The accuracy is in full agreement with that predicted in Theorem \ref{teo:esterrorgrad}. 
\begin{figure}[ht]
	\centering
	\includegraphics[width=0.4\textwidth]{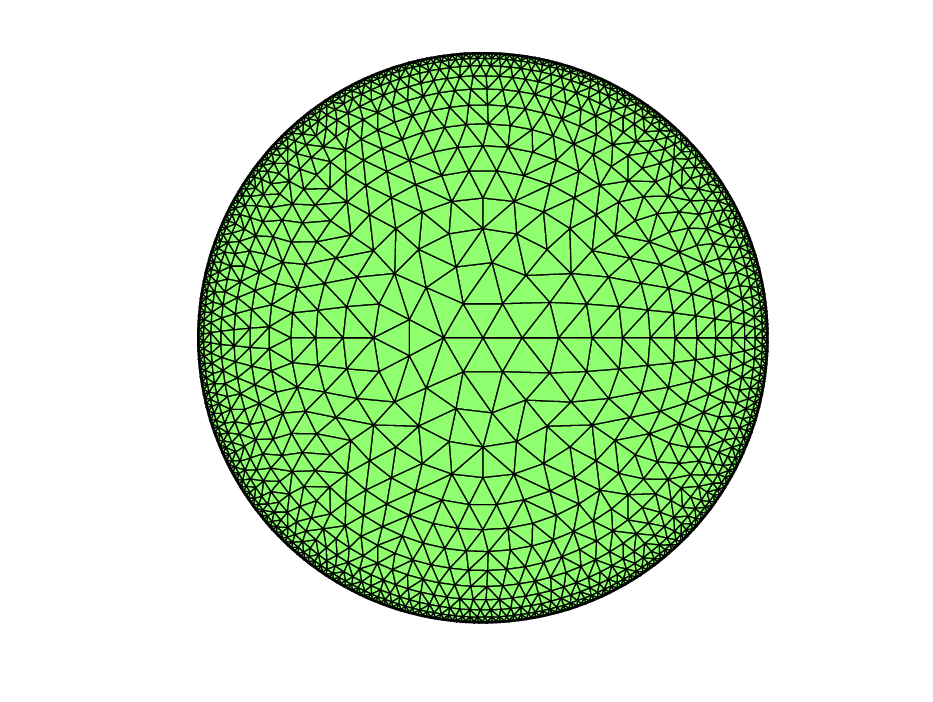} 
	\includegraphics[width=0.4\textwidth]{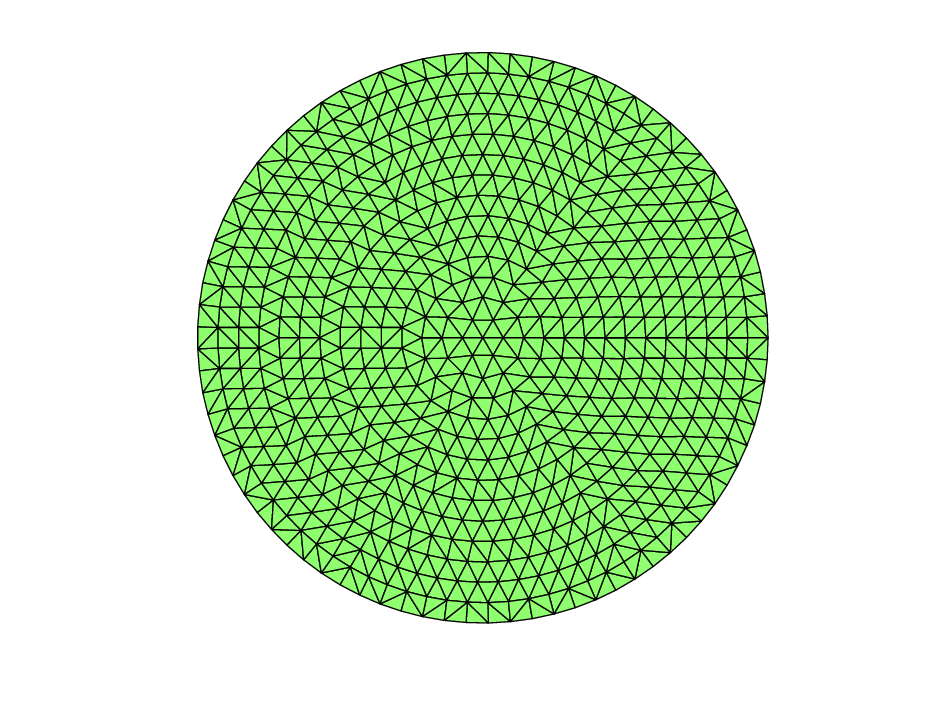} 
	\label{fig:ejemplograd}
	\caption{Left: graded mesh with $M=15$ and $\mu=2-\varepsilon$. Right: uniform mesh with $M=15$ and $\mu =1$.}
\end{figure}

\begin{table}[h] \footnotesize\centering
\begin{tabular}{|c| c|} \hline
		Value of $s$ & Order (in $h$) \\ \hline
		0.5 & 1.066 \\
		0.6 & 1.051 \\
		0.7 & 0.990 \\
		0.8 & 0.985 \\
		0.9 & 0.977 \\
 \hline	 
\end{tabular}
\caption{(Graded Meshes) Rates of convergence in the norm $\| \cdot \|_\V$ for problem \eqref{eq:ejemplo} and $s \geq 1/2$. The mesh parameter $h$ behaves like $N^{-1/2}$, $N$ being the number of nodes.}
\label{tab:ejemplo3}
\end{table}
\begin{remark}\label{rem:mallas}
Taking into account the restrictions \eqref{eq:cotasReg}, it is possible to achieve differentiability orders between $1/2+s< \ell < 2$ by choosing adequate weights. At this point, the reader might ask whether the order of convergence (with respect to $N$) could be improved by considering a different value of $\ell$ and following the grading approach presented at the beginning of this subsection. This is not the case; actually the choice we made yields the best possible order w.r.t. the number of nodes with minimum grading requirements on the mesh.

Indeed, it is simple to check that, for a given regularity $\ell$,  the optimal choice for the grading parameter is $\mu = 2 (\ell - s)$. Recall identity \eqref{eq:conysinpeso} and the results shown in Section \ref{ss:fem}, which give
$$\| u - u_h \|_\V \leq C h^{\ell - s} |u|_{H^{\ell}_{\alpha}(\W)} , $$
where $h$ is the mesh parameter.

If we restrict to $\ell < 1 + s$, then $\mu < 2$ and the number of nodes is $N \sim h^{-2}$. Therefore, as $\ell$ increases there is a gain of order without an increment in the total number of nodes and the error behaves like $N^{-(\ell - s)/2}$. Within this range, the choice $\ell = 1 + s - \eps$ is optimal.

On the other hand, if we consider $\ell > 1 + s$ then $\mu > 2$ and it is simple to check that in this case $N \sim h^{-\mu}$. Here the gain of order one might expect due to the increase in differentiability is compensated by the cost of having to increase the weight power, which implies a growth in the number of nodes. In the whole range $\ell \in (1+s, 2)$ we obtain that the error behaves like $N^{-1/2}  \ln N$.
\end{remark}

\section{Conclusion} In this paper, a complete Finite Element study of  a fractional Laplace equation is carried out. First it is shown that recent H\"older regularity results for this problem  \cite{RosOtonSerra} can be used to provide a priori estimates in weighted fractional Sobolev spaces, within which the FE setting can be straightforwardly adapted. In particular, some of these estimates measure in a precise way the singular behavior of solutions near the boundary. Borrowing techniques from the BEM it was found that the singular kernel arising in this problem can be accurately handled.  The FE method is implemented in one and two dimensions, where uniform as well as tailored graded meshes are proposed. Error estimates for the Scott-Zhang interpolation operator in fractional weighted spaces are obtained by introducing an appropriate version of the improved Poincar\'e inequality.  These error estimates are used to prove optimal order of convergence of the FE method in the weighted fractional context. Numerical experiments are presented  delivering orders of convergence in full agreement with our theoretical predictions.

\bibliography{bib}
\bibliographystyle{siam}
\end{document}